\documentclass[a4paper,11pt]{article}
\usepackage{amsmath}
\usepackage{amssymb}
\usepackage{theorem}
\usepackage{pstricks,url}
\usepackage{euscript}
\usepackage{epic,eepic}
\usepackage{graphicx}
\usepackage{multirow,colortbl,hhline}
\usepackage{subfigure}
\PassOptionsToPackage{normalem}{ulem}
\usepackage{amsmath}
\usepackage{amssymb}
\usepackage{theorem}
\usepackage{pstricks}
\usepackage{euscript}
\usepackage{epic,eepic}
\usepackage{graphicx}
\usepackage{setspace}
\PassOptionsToPackage{normalem}{ulem}

\newcommand{\norm}[1]{\|#1\|}
\newcommand{\what}{\widehat{w}}
\newcommand{\vhat}{\widehat{v}}

\newcommand{\RR}{\ensuremath{\mathbb R}}

\newcommand{\NN}{\ensuremath{\mathbb N}}

\newcommand{\prox}{\ensuremath{\operatorname{prox}}}

\newcommand{\argmin}{\ensuremath{\operatorname{argmin}}}

\newtheorem{theorem}{Theorem}[section]
\newtheorem{lemma}[theorem]{Lemma}

\newtheorem{proposition}[theorem]{Proposition}

\theoremstyle{plain}{\theorembodyfont{\rmfamily}
}
\theoremstyle{plain}{\theorembodyfont{\rmfamily}
}
\theoremstyle{plain}{\theorembodyfont{\rmfamily}
}
\theoremstyle{plain}{\theorembodyfont{\rmfamily}
}
\theoremstyle{plain}{\theorembodyfont{\rmfamily}
}
\theoremstyle{plain}{\theorembodyfont{\rmfamily}
\newtheorem{remark}[theorem]{Remark}}
\theoremstyle{plain}{\theorembodyfont{\rmfamily}
}

\definecolor{labelkey}{rgb}{0,0.08,0.45}
\definecolor{refkey}{rgb}{0,0.6,0.0}
\definecolor{Brown}{rgb}{0.45,0.0,0.05}
\definecolor{dgreen}{rgb}{0.00,0.49,0.00}
\definecolor{dblue}{rgb}{0,0.08,0.75}
\numberwithin{equation}{section}

\begin{document}
\title{\sffamily\huge 
Don't relax: early stopping for  convex regularization
}
\author{Simon Matet $^1$, Lorenzo Rosasco$^{2,3}$, Silvia Villa$^4$ and  B$\grave{\text{\u{a}}}$ng C\^ong V\~u$^5$\\ 
\small$^1$ Ecole Politechnique, Route de Saclay\\
\small 91128 Palaiseau, Cedex, France\\
\small$^2$ LCSL, Istituto Italiano di Tecnologia and Massachusetts Institute of Technology\\
\small Bldg. 46-5155, 77 Massachusetts Avenue, Cambridge, MA 02139, USA\\
\small $^3$ DIBRIS, Universit\`a di Genova, Via Dodecaneso 35\\ 
\small 16146 Genova, Italy\\
\small $^4$ Dipartimento di Matematica, Politecnico di Milano, Via Bonardi 9\\
\small 20133 Milano, Italy\\ 
\small $^5$ Laboratory for Information and Inference Systems, EPFL, ELD 243 (Batiment EL) \\
Station 11, CH-1015, Lausanne, Switzerland\\
\url{simonmatet@hotmail.com};\; \url{lrosasco@mit.edu};\\ \url{silvia.villa@polimi.it};\;\url{cong.bang@iit.it}
}
\maketitle
\begin{abstract}
We consider the problem of designing efficient regularization algorithms
when regularization is encoded by a (strongly) convex functional.
Unlike classical penalization methods based on a relaxation approach, we 
propose an iterative method where regularization is achieved via early stopping.
Our results show that the proposed procedure achieves the same recovery accuracy as
penalization methods, while naturally integrating computational considerations. 
An empirical analysis on a number of problems provides promising results with respect to the state of the art.
\end{abstract}

{\bf Keywords:} 
monotone inclusion,
maximal monotone operator,
operator splitting,
cocoercive operator,
composite operator,
duality,
stochastic errors,
primal-dual algorithm
 
 {\bf Mathematics Subject Classifications (2010)}: 47H05, 49M29, 49M27, 90C25 

\section{Introduction}

Many machine learning problems require  to estimate  a quantity of interest based on  random noisy data/measurements.  Towards this end,   a common  approach is considering estimators defined by the minimization of an empirical objective, where a data fit  term is penalized using {\em a regularizer}, encoding prior information on the  quantity  to be estimated. From a modeling perspective, this latter approach can be seen   as the relaxation of an ideal problem with equality constraints defined by exact data, whereas from a computational perspective it reduces in principle to the solution of a single  optimization problem. 

In practice however, the regularization parameter needs to be chosen, and hence the solution of multiple optimization problems is typically required. Moreover, computational and estimation aspects are usually considered separately, leading to  potential dichotomy and trade-offs between estimation and computational aspects \cite{BouBou11}. 
Indeed, these observations have recently motivated the development of techniques  to compute solutions corresponding to different penalization levels (regularization path) \cite{EfrHasJoh04,HasRosTib04} as well as an interest on the interplay between estimation and computation \cite{OymSolRec15}. 
 
In this paper, we investigate and apply   iterative regularization techniques in the context of linear inverse problems modeling many machine learning problems.  The key idea behind iterative regularization is that early stopping the iterative  optimization of 
 an empirical  problem,  performs a form of implicit regularization \cite{EngHanNeu96}.
Iterative regularization algorithms  are classical in inverse problems \cite{EngHanNeu96} and  have been recently analyzed  and   applied  in machine learning to finf the minimal norm solution \cite{BauPerRos07,BlaKra10,RasWaiYu14,RosVil15,YaoRosCap08,ZhaYu05}.
These works show that iterative regularization methods typically share the same estimation 
properties of penalized methods, but are often advantageous from a computational perspective. 
Indeed, since  the number of iterations becomes the regularization parameter, 
iterative regularization schemes have a built-in warm restart property that allows to
easily compute a whole regularization path, if the involved regularizer is the squared norm one.

The main question we discuss  in this work is how to derive and analyze fast    
 iterative regularization schemes  for large classes of  regularizers. Indeed,  flexibility in the choice 
 of this latter  functional  is  key for good estimation and has been the subject of much recent work. 
However, while how to exploit such penalties is clear using relaxation approaches, how to derive corresponding iterative regularization schemes is less obvious.

In this paper,  we  derive  iterative regularization for a strongly convex regularizer, by considering the iterative minimization of this latter  functional under  
equality  constraints defined by the {\em noisy data} (rather than a relaxation). The iteration thus obtained {\em does not} converge to the desired solution, but can be shown 
to be robust to noise if suitably stopped. Indeed, a stability argument shows the noisy iteration deviates gradually from a noiseless iteration  which in turns can be shown to 
converge  to the ideal solution. Exploiting this latter result, an optimal stopping rule  and the corresponding recovery results can be  derived.

 We explore 
this general idea considering two distinct iterations. The first is based on a dual gradient descent (a.k.a. mirror descent \cite{BecTeb03}, 
and linearized Bregman iteration \cite{BurOsh13}), while the second corresponds  to an accelerated variant  \cite{BecTeb14}. While both methods are 
shown to lead to the same recovery guarantees, acceleration  allows for more aggressive stopping rules with substantial computational gains. 

 The idea of considering iterative regularization and early stopping for convex regularizers is not new, we refer to \cite{BurSawSte14} for an 
 interesting survey on known results, open problems, and additional references. Some previous approaches \cite{OshBurGol05} rely 
 on Morozov discrepancy principle \cite{EngHanNeu96}, other approaches are based on stability, see \cite{BurResHe07,BotHei12}. 
 However the existing studies do not analyze the algorithms presented in this paper.
More importantly, we are not aware of any previous results considering the regularization effect of accelerated iterations \cite{Nes07}. 

Our theoretical findings are complemented by empirical results on three different applications:
variable selection, matrix completion, and image deblurring. The experiments  confirm
the theoretical results and show that the recovery properties of iterative regularization 
are comparable to penalization approaches with  much lower computational costs.

The rest of the paper is organized as follows: in Section~\ref{sec:ps} 
we describe the setting and the main assumptions, in Section~\ref{sec:S2} we introduce
the iterations we study, and in Section~\ref{sec:main}  we state the main results, 
discuss them, and provide the main elements of the proof. 
In Section~\ref{sec:exp} we present several experimental results on matrix completion, variable selection, and 
deblurring problems.

\section{Problem setting}
\label{sec:ps}
We consider a general problem of the form 
\begin{equation}
y=Xw^\dag, 
\end{equation}
for a given matrix\footnote{
For simplicity, the results are stated in finite dimensional euclidean spaces, but all the  conclusions hold if $\mathbb{R}^p$ and $\mathbb{R}^n$ are replaced  by
 Hilbert spaces $\mathcal{H}$ and $\mathcal{G}$.}
 $X\colon\RR^p\to\RR^n$, an observation $y\in\RR^n$, and a vector $w^\dagger\in\RR^p$. 
Such formulation include for instance  regression, feature selection,  as well as many  image/signal processing   problems.  In general, the solution of the
above linear equation is not unique, and a selection  principle is needed to choose an appropriate  solution (e.g.  in the high dimensional scenario, where $p>n$).
In this paper we assume that the solution of interest $w^\dag$ minimizes a  function  $R\colon\RR^p\to \left]-\infty,+\infty \right]$
 encoding some prior information  on the  problem at hand. 
We assume $R$ to be proper,
lower semicontinuous, strongly convex,
and we let  $w^\dag$ to be the unique solution of the optimization problem
\begin{equation}
\label{e:main}
\underset{ y=Xw}{\text{minimize}}\; R(w).
\end{equation}
In practice, one does not have access to $y$, but only to
a noisy version $\widehat{y}$. In particular in this paper we consider a
worst case scenario, where the noise is deterministic, i.e.   $\|y-\widehat{y}\|\leq \delta$, for some  $\delta>0$.
The goal is then to find a stable estimation of  $w^\dag$ only observing $X$ and $\widehat{y}$.

The classical way to achieve this goal is to relax the equality constraints, and use a Tikhonov regularization scheme:
\[
\min_{w\in\mathbb{R}^p}\|\widehat{y}-Xw\|^2 + \lambda R(w).
\]
A data fidelity term is added to the function $R$, multiplied by a regularization parameter $\lambda$.
Such an approach usually requires two steps: first, the solution of a regularized problem for several 
values of the regularizing parameter, and second the model selection, where the best regularized solution
is selected among the computed ones. 
 
In this paper we avoid  relaxation, and consider iterative regularization schemes.  
We define a sequence $(\widehat w_t)_{t\in\NN}$  derived by applying an appropriate minimization algorithm 
to the noisy problem 
\begin{equation}
\label{e:noisy}
\underset{\widehat{y}=Xw}{\text{minimize}}\; R(w).
\end{equation}
Such a procedure converges to a minimizer of the noisy problem~\eqref{e:noisy}, which is not 
the solution we are looking for, however a good solution can be achieved by  early stopping. More precisely,
we show that, depending on the noise level, we can select an element $\widehat w_{t_\delta}$ 
of the sequence $(\widehat w_t)_{t\in\NN}$  which  converges  to $w^\dag$  when the noise goes to zero.
An intuition of why this procedure works can be derived from the proof's strategy. To analyze the behavior of
the sequence $(\widehat w_t)_{t\in\NN}$ we define an auxiliary (regularizing) sequence $(w_t)_{t\in\NN}$,  
that is the sequence obtained applying the same minimization algorithm devised for problem \eqref{e:noisy},
to the ideal problem~\eqref{e:main}, which therefore converges to $w^{\dag}$. The choice of the
stopping time will be derived by the following error decomposition
 $$
 \norm{\widehat w _t -w^\dagger}\le  \norm{\widehat w _t -w_t}+ \norm{w _t -w^\dagger}.
 $$
The term $\norm{w _t -w^\dagger}$ is an optimization (or regularization) error.
 We will show that it vanishes for increasing $t$ and in fact will prove non asymptotic bounds. 
 The term $\norm{\widehat w _t -w _t}$ measures stability to noise and we will see to  increase with  $t$
 and $\delta$.
Given data and knowledge of the  noise level, our actual regularization procedure is specified by a suitable choice
 $ t_\delta$ and this results in the explicit  bound 
$\norm{\widehat w _{t _\delta}-w ^\dagger}\leq c \delta^{1/2}$. Note that the dependence on the noise level 
$\delta$ is the same as in Tikhonov regularization \cite{EngHanNeu96}.
In the rest of the paper, we develop the above idea providing all the details.

\noindent{\bf Notation} In the following, the operator norm of the
matrix $X$ is denoted by $\|X\|$.
 
\section{Iterative regularization algorithm for a general penalty}
\label{sec:S2}
In this section we begin presenting the iterative regularization procedures we study based on dual gradient descent (DGD)
and accelerated dual gradient descent (ADGD). The first one is
a basic algorithm, while the second is its accelerated version, requiring some additional
steps.  First, recall that the regularizing function $R$ in~\eqref{e:noisy} is assumed to be strongly convex. 
This implies that there exists $\alpha\in\left]0,+\infty\right[$ 
and a proper, lower semicontinuous, and convex function $F\colon\RR^p\to[0,+\infty]$
such that
\begin{equation}
\label{e:R}
R=F+\frac{\alpha}{2}\|\cdot\|^2.
\end{equation}

Both DGD  and ADGD belong to the class of first order methods, requiring only matrix and vector multiplications, and
the computation of the proximity operator of $\alpha^{-1}F$, which is defined as
\begin{align}
\label{e:prox}
(\forall w\in\RR^{p})\qquad\prox_{\alpha^{-1} F}(w)= \argmin_{u\in\RR^p}\Big\{ F(u)+\frac{\alpha}{2} \|u-w \|^2  \Big \}. 
\end{align}
The computation of the proximity operator involves a minimization problem, which  can be solved explicitly in many
relevant cases \cite{ComPes11}. In particular, it reduces to the well-known soft-thresholding operator when $F$ is equal 
to the $\ell^1$ norm, and to a projection, when $F$ is the indicator function of a convex and closed set. 
We will show in the supplementary material that DGD reduces to a gradient descent on the dual of problem in~\eqref{e:noisy}. 
Its asymptotic minimization properties for the problem in~\eqref{e:noisy}, which is not the one we want to solve,
have been studied in~\cite{Svva1}. Note that this algorithm, up to a change of variables, is called linearized Bregman 
iteration in the series of papers \cite{OshBurGol05,BacBur09,ZhaBurBre10,BurOsh13,BurSawSte14}. The same algorithm
is also called mirror descent in the optimization community \cite{BecTeb03}.
By considering a Nesterov acceleration \cite{Nes07} of gradient descent, we derive ADGD, that is the FISTA
variant on the dual problem, which has been considered in~\cite{BecTeb14,VilSal13}.
Algorithms DGD and ADGD can be seen as minimization algorithms applied to the dual of the original noise free 
problem in~\eqref{e:main}, in the presence of a nonvanishing error on the gradient.

\begin{center}
\begin{tabular}{cc}
\begin{tabular}{ l } 
\hline
{\bf Dual Gradient Descent (DGD)}\\
\hline
\\[-0.8ex]
 Let $\vhat_0=0\in\RR^p$  and  $\gamma=\alpha\|X\|^{-2}$\\[0.6ex]
 For $t=0,1,\ldots$ iterate\\[0.6ex]
\hspace{0.5cm}$\what_{t} = \prox_{\alpha^{-1}F}\big(-\alpha^{-1} X^T\vhat _t\big)$\\
\hspace{0.5cm}$\vhat_{t+1} = \vhat_t + \gamma(X \what_{t} -\widehat{y})$\\
\hspace{0.5cm}$\widehat{u}_t=\dfrac{1}{t+1}\sum_{k=0}^t\what_k$\\
\\
\\[-0.8ex]
\hline
\end{tabular}
&
\begin{tabular}{ l } 
\hline
{\bf Accelerated Dual Gradient Descent (ADGD)}\\
\hline
\\[-0.8ex]
 Let $\vhat_0=\widehat{z}_{-1}=\widehat{z}_0\in\RR^p$, $\gamma=\alpha\|X\|^{-2}$, and $\theta_0=1$\\[0.6ex]
 For $t=0,1,\ldots$ iterate\\[0.6ex]
\hspace{0.5cm}$\what_{t} = \prox_{\alpha^{-1}F}\big(-\alpha^{-1} X^T\widehat{z}_t\big)$\\
\hspace{0.5cm}$\widehat{r}_{t} = \prox_{\alpha^{-1}F}\big(-\alpha^{-1} X^T\vhat _t\big)$\\
\hspace{0.5cm}$\widehat{z}_{t} = \vhat_t + \gamma( X \widehat{r}_{t} -\widehat{y}) $\\
\hspace{0.5cm}$\theta_{t+1}=(1+\sqrt{1+4\theta_t^2})/{2}$\\
\hspace{0.5cm}$\vhat_{t+1}=\widehat{z}_t+\frac{\theta_t-1}{\theta_{t+1}}(\widehat{z}_t-\widehat{z}_{t-1})$\\
\hline
\end{tabular}
\end{tabular}
\end{center}

%
%
%

Before studying the regularizing properties of the proposed procedures, we show that DGD
is a generalization of the well-known Landweber iteration (see \cite{EngHanNeu96}). 

\begin{remark}[Connections with Landweber iteration]
Consider Algorithm DGD in the special case $F=0$. Noting that, for every $w\in\RR^p$, $\prox_{\alpha^{-1}F}(w)=w$, 
we derive
\begin{equation}
\what_{t+1}=\what_{t}-\gamma\alpha^{-1}X^T(X\what_{t}-\widehat{y}),
\end{equation}
which coincides with the Landweber iteration for solving Problem~\ref{e:main},
studied in the context of regression in \cite{YaoRosCap08}. 
ADGD provides a FISTA variant of Landweber iteration, for which we prove here regularization properties. 
\end{remark}

The previous remark shows that the proposed algorithms are generalization of the Landweber iteration for a more general penalty term
of the form in \eqref{e:R}. 
While it is well known that early stopping of the Landweber iteration leads to stable approximations of the minimal norm solution
of an inverse problem, here we generalize such result to obtain stable approximations of the solution defined by 
general regularizers. The presence of the additional term $F$ in the regularization function introduces in the algorithm 
a (nonlinear) proximal operation. 

\section{Early stopping for strongly convex iterative regularization}
\label{sec:main}

In this section, we present and discuss the main results of the paper.
We start with DGD.

\begin{theorem}[Dual gradient descent]
\label{thm:dgd}
Let $\delta\in\left]0,1\right]$.
Let   $(\widehat{u}_t)_{t\in\NN}$  be  the averaged
sequence generated by DGD.
Assume that there exists $\bar{v}\in\RR^p$ such that
$-X^T\overline{v}\in\partial R(w^\dag)$.
Set $a=2\|X\|^{-1}$ and $b=\|X\| \|v^\dag\|\alpha^{-1}$, where
$v^\dag$ is a solution of the dual problem of~\eqref{e:main}.
Then, for every $t\in\NN$, 
\begin{equation}
\label{e:aqavg}
\|\widehat{u}_{t} -w^\dag\|  \leq a  t^{1/2}\delta + b{t}^{-1/2}.
\end{equation}
In particular, choosing $t_\delta=\lceil c\delta^{-1}\rceil$ for some $c>0$, we derive
\begin{equation}
\label{eq:bdgd}
\|\widehat{u}_{t_\delta} -w^\dag\|  \leq  \big[a(c^{1/2}+1)+bc^{-1/2}\big] \delta^{1/2}.
\end{equation}
\end{theorem}
Before discussing the above result, we state an analogous result for the accelerated variant. 
\begin{theorem}[Accelerated dual gradient descent] 
\label{thm:adgm}
Let $\delta\in\left]0,1\right]$ and let  $(\widehat{w}_t)_{t\in\NN}$  be  the
sequence generated by ADGD. 
Assume that there exists $\bar{v}\in\mathbb{R}^p$ such that
$-X^T\overline{v}\in\partial R(w^\dag)$.
Set $a=4\|X\|^{-1}$ and $b=2\|X\| \|v^\dag\|/\alpha$, where $v^\dag$ is a solution of the dual 
problem of~\eqref{e:main}.
Then, for every $t\geq 2$,
\begin{equation}
\label{e:aq}
\| \what_t -w^\dag\|  \leq a t\delta + b{t}^{-1}.
\end{equation}
In particular, choosing $t_\delta=\lceil c \delta^{-1/2}\rceil $ for some $c>0$,
\begin{equation}
\label{e:aq2a}
\| \what_t -w^\dag\|  \leq  \big[a(c+1)+bc^{-1}\big] \delta^{1/2}.
\end{equation}
\end{theorem}
We first discuss the results and make a comparison with related work, and then give a sketch of the proof. 
The complete proof can be found in the supplementary material. 

\paragraph{Discussion and comparison with related work}
As anticipated in Section~\ref{sec:ps}, the bounds in \eqref{eq:bdgd} and \eqref{e:aq2a} are derived
by optimizing a stability plus regularization/optimization bound. Note  in particular that  the constants 
appearing in the  regularization error are determined by the strong convexity constant and the norm 
of the operator $X$. The above results show that, 
given a noise level $\delta$,  regularization is achieved computing a suitable number 
$t_\delta$ of iterations of DGD and ADGD. The number of required iterations 
tends to infinity  as the noise goes to zero. 
The definition of $t_\delta$ in Theorems~\ref{thm:dgd} and \ref{thm:adgm} is an 
early stopping rule.  The dependence of the noise that we get in Theorems~\ref{thm:dgd} and \ref{thm:adgm} is
optimal \cite{EngHanNeu96}, and coincides with the Tikhonov regularization one. 
The difference between DGD and ADGD is on the computational
aspect: indeed, to achieve the same recovery accuracy, a number of iterations of the order of $\delta^{-1}$ are needed
for the basic scheme, and only $\delta^{-1/2}$ iterations are needed for the accelerated method. 
This kind of result resembles the  behaviour of the $\nu$-method for the minimal norm solution \cite{EngHanNeu96}.

The condition $-X^T\overline{v}\in\partial R(v^\dag)$ can be interpreted as an abstract regularity 
condition on the subdifferential of $R$ \cite{BurSawSte14}. When $R=\|\cdot\|^2/2$, and more generally when $R$ is real-valued, it is 
 automatically satisfied under our assumptions,  and it corresponds to what is called a source condition \cite{EngHanNeu96,CapDev07}. 

\begin{remark}[Avoiding averaging]
For DGD, regularizing properties are proved for the averaged sequence. 
However, if sparsity properties of the solution are of interest, averaging is not appropriate.
For the nonaveraged sequence $\what_t$ defined by DGD, we have that
for every $\delta\in\left]0,1\right]$ there exists $t_\delta=O(\delta^{-1})$ 
such that $\| \what_{t_\delta} -w^\dag\|  \leq (a+2 b) \delta^{1/2}$,
with $a$ and $b$ defined as in Theorem~\ref{thm:dgd}. See Proposition~\ref{l:s1} and Theorem~\ref{thm:dupri} in the supplementary material for the proof. 
\end{remark}

\begin{remark}[Inexact prox]
In some interesting cases, the proximity operator is not available in closed form, but can be still  computed 
at reasonable cost (see \cite{SalVil12,VilSal13} for a throughout discussion). The results in Theorem~\ref{thm:dgd}
and \ref{thm:adgm} hold also if the proximity operator is computed inexactly, at an increasing precision.    
\end{remark}

\begin{remark}[Beyond worst case]
While we considered a general regularization $R$ and obtained worst-case results, 
an interesting question is if these results can be improved under additional assumptions on $R$,
e.g. assuming it is sparsity inducing. This will be the subject of future work, and
we refer to \cite{OshRuaXio14} for some results in this direction. 
\end{remark}

We  next compare our iterative regularization methods with  related work. 
The case $R=\|\cdot\|^2$ is classic, see~\cite{EngHanNeu96}.
In \cite{OshBurGol05} an iterative regularization procedure based on the so called Bregman iteration, 
is considered. An early stopping rule based on a discrepancy principle in the case of noisy data is also presented. 
There is one main difference with respect to our contribution. Each DGD or ADGD step does not require inner algorithms
if the proximity operator is available in closed form, while Bregman iteration requires the solution of a nontrivial
minimization problem at each step. Such step is computationally as costly as solving a Tikhonov regularized problem. 
A stability analysis for the Bregman iteration is presented in Theorem~4.2 in \cite{BurResHe07}, while
weak convergence without the strong convexity assumption is proved in \cite{OshBurGol05}.
A qualitative early stopping  rule for the DGD  algorithm has been considered in \cite{BacBur09} for the total variation 
case. Finally, a related algorithm to the DGD  is the one considered in \cite{BotHei12}. The setting of \cite{BotHei12} is more general
than ours, but the obtained results  are weaker: the stopping rule is of the form $O(\delta^{-2})$ and no quantitative
bounds of $\|\widehat{w}_{t_\delta}-w^\dagger\|$ are given. 

\paragraph{Sketch of the proof}
\label{sec:S1}
We now discuss the main elements of the proof. The complete argument can be found in the supplementary material. 
We start from the proof of Theorem~\ref{thm:dgd} and then we will briefly comment on the proof of Theorem~\ref{thm:adgm}.
The proof of Theorem~\ref{thm:dgd} is based on a decomposition of the error to
be estimated in two terms. The idea is  to build an auxiliary sequence and 
to majorize the error with the sum of two quantities that can be interpreted as 
a stability and an optimization (regularization) error, respectively.  
Bounds on these two terms are then provided.  
We first introduce the corresponding algorithm to solve the target problem in \eqref{e:main}. 
This algorithm is not  used in  practice, but is needed only for the theoretical analysis, 
and is the noise free version of DGD, where $\widehat{y}$ is replaced 
by $y$.  Starting from $v_0=0$, the $t$-th iteration is defined by
\begin{align}
\label{e:main1}
w_{t} = \prox_{\alpha^{-1}F}\big(-\alpha^{-1} X^Tv_t\big),\quad v_{t+1} = v_t + \gamma(X w_{t} -y), \quad u_t=\sum_{k=0}^{t}w_t/(t+1)
\end{align}
for the  gradient descent algorithm applied to the dual 
of problem~\eqref{e:main} (see the supplementary material for its definition). 
The choice of  $ t_\delta$ is derived from the 
the following error decomposition
 $$
 \norm{\widehat u _t -w^\dagger}\le  \norm{\widehat u _t -u_t}+ \norm{u _t -w^\dagger}.
 $$
The term $\norm{u_t -w^+}$ is called approximation, but also optimization or regularization error.
It vanishes for increasing $t$ and in fact the following non asymptotic bound holds
$ \norm{u_t -w^\dagger}\leq {\|X\|\|v^\dagger\|}{\alpha^{-1}t^{-1/2}}$. 
The term $\norm{\widehat u_t -u_t}$ measures stability and
its behavior for fixed $t$ and noise level $\delta$ is 
$\norm{\widehat u_t -u_t}\le 2\|X\|^{-1} \delta t^{1/2}.$
The choice of $t_\delta$ is obtained optimizing the resulting  bound with respect to 
$t\in\NN$, that is $t_\delta= \argmin_{t\in\mathbb{N}}\left( \|X\|\|v^\dagger\|\alpha^{-1}t^{-1/2}+2\|X\|^{-1} \delta t^{1/2} \right)$.

The stopping rule for ADGD follows analogously from a general result about convergence of proximal methods in the presence of
computational errors \cite{AujDos15}.
%

 \section{Numerical experiments}
 \label{sec:exp}
  In this section we compare our iterative regularization techniques (DGD and ADGD with early stopping) with 
  Tikhonov regularization  on three different problems: variable selection, matrix completion, and image deblurring. 
The performance of the Tikhonov regularization scheme depends of course on the chosen algorithm to solve the regularized problems.
We use state of the art techinques: accelerated proximal gradient descent with warm-restart \cite{HalYinZha08}. 
The model selection phase is performed as follows:
we first solve the regularized problem with a very large value $\lambda_{0}$, and then for the sequence
$\lambda_i=2^{-i}\lambda_0 $. Since in practice the noise level is unknown, we choose $\lambda$ using  holdout cross-validation
keeping 1/10 of the available points for validation.
For initializing the accelerated gradient descent on the regularized problem we use the warm-restarting trick,  which is known
(in practice) to dramatically accelerate the computation of the regularization path \cite{BecBobCan11}. 
The comparison relies heavily  on the stopping rule used for stopping the iteration computing the minimizer of the Tikhonov 
regularized functional. We used a very loose stopping rule for the algorithm for a given $\lambda_i$ to make Tikhonov regularization
more competitive. More precisely the iterations were stopped when the distance between to successive iterations was less than $0.001\cdot\delta$. 
Since accelerated proximal gradient descent involves steps with the same computational complexity  to those of DGD and ADGD, the comparison 
between the three approaches is made in terms of number of iterations. The number of iterations for Tikhonov regulatization is the total number of iterations
for all different $\lambda$ values.  
 \subsection{Variable selection}
We consider a linear regression problem with $n=500$ examples and $p=2000$ variables. 
We assume that  $\widehat{y}\in\RR^{500}$ is obtained corrupting  with a Gaussian noise of
mean zero and variance $\delta/\sqrt{n}$ a measurement $Xw_*$, where $w_*$ is a vector having a small number
of nonzero components (10, 30, or 60, respectively). In this example, the covariates are correlated with a random 
covariance matrix $\Sigma$ with $\Sigma=C^TC$, where $C$ is a random matrix
 with entries drawn independently at random from a gaussian distribution with standard deviation $0.1$.
To perform variable selection, and obtain a sparse estimator 
we apply our iterative regularization methods, DGD and ADGD, to the 
elastic net regularizing function $R(w)=\|w\|_1+({\alpha}/{2})\|w\|^2$.
 We compared the number of iterations of DGD, ADGD, and Tikhonov regularization on 50 different realizations of sample points. 
 The parameters were chosen using a validation set of 100 samples. 
  The results are shown in Table~\ref{feat-table}. 
 For Tikhonov regularization, we used a second  least squares step on the selected variables to compute the validation score, 
 requiring an extra computation load that we did not quantify here. It is worth noticing that iterative regularization does not require this further step. 
The results suggest that Tikhonov regularization and iterative regularization algorithms have very similar prediction and variable selection 
performances. DGD is approximately as fast as state of the art variational regularization, while ADGD is much faster. 

\begin{table}[t]
 { \caption{\small Performances of DGD, ADGD, and warm-started Tikhonov regularization with accelerated proximal gradient descent. False positives are the selected irrelevant variables. False negatives are the discarded relevant features. Prediction error is the average prediction error of the estimated solution in percent. The results are averaged over 50 trials with the standard deviation between parentheses.}
\label{feat-table}
\begin{center}
{\footnotesize
\begin{tabular}{ cclcccc } 
\hline
\multirow{2}{2em}{Noise} & Relevant  & \multirow{2}{4em}{Algorithm}& False & False &Prediction &\multirow{2}{4em}{Iterations} \\
& Variables & & Positive & Negative& Error &\\
\hline
\multirow{9}{2em}{0.1} & \multirow{3}{2em}{10}& \cellcolor[gray]{0.9} DGD &\cellcolor[gray]{0.9}0.10 (0.3)&\cellcolor[gray]{0.9}0.53 (0.7) &\cellcolor[gray]{0.9}3.7 (0.6) &\cellcolor[gray]{0.9}890 (200)\\
&&{\cellcolor[gray]{0.8}} ADGD &\cellcolor[gray]{0.8}0.40 (0.9)&\cellcolor[gray]{0.8}0.53 (0.7) &\cellcolor[gray]{0.8}3.7 (0.6) &\cellcolor[gray]{0.8}140 (30)\\ 
&& Tikhonov & 0.62 (1)&0.28 (0.5) &3.6 (0.3) & 580 (40)\\ 
\hhline{|~------|}
&\multirow{3}{2em} {30} & {\cellcolor[gray]{0.9}} DGD &{\cellcolor[gray]{0.9}}8.8 (5) &{\cellcolor[gray]{0.9}}1.8 (1) &{\cellcolor[gray]{0.9}} 4.8 (0.4)&  {\cellcolor[gray]{0.9}}860 (90)\\
&&\cellcolor[gray]{0.8} ADGD &\cellcolor[gray]{0.8}5.0 (5)&\cellcolor[gray]{0.8}1.8 (1) &\cellcolor[gray]{0.8}4.6 (0.4) &\cellcolor[gray]{0.8} 110 (16)\\ 
&& Tikhonov & 12 (9)&  2.1 (1)&5.4 (0.6) & 860(140)\\ 
\hhline{|~------|}
&\multirow{3}{2em} {60}&  \cellcolor[gray]{0.9 } DGD & \cellcolor[gray]{0.9}49 (10)&\cellcolor[gray]{0.9}5.2 (2) &\cellcolor[gray]{0.9} 8.1 (0.8) &\cellcolor[gray]{0.9} 940 (100)\\
&&\cellcolor[gray]{0.8}  ADGD &\cellcolor[gray]{0.8}  27 (10)&\cellcolor[gray]{0.8} 5.7 (2) &\cellcolor[gray]{0.8}  7.4\phantom{0}(0.7) &\cellcolor[gray]{0.8}  170 (30)\\ 
&&Tikhonov & 53 (20)& 5.7 (2) & 7.4 (0.7) & 1800 (400)\\ 
\hline
\multirow{9}{2em}{1} & \multirow{3}{2em}{10}& \cellcolor[gray]{0.9} DGD &\cellcolor[gray]{0.9} 2.2 (3)&\cellcolor[gray]{0.9} 2.7 (1) &\cellcolor[gray]{0.9} 46 (2) &\cellcolor[gray]{0.9} 480 (100)\\
&&{\cellcolor[gray]{0.8}} ADGD &\cellcolor[gray]{0.8} 1.1 (2)&\cellcolor[gray]{0.8} 2.8 (1) &\cellcolor[gray]{0.8} 45 (2) &\cellcolor[gray]{0.8} 92 (40)\\ 
&& Tikhonov & 2.3 (2)& 2.9 (1) & 48 (4) & 360 (90)\\ 
\hhline{|~------|}
&\multirow{3}{2em} {30} & {\cellcolor[gray]{0.9}} DGD &{\cellcolor[gray]{0.9}} 17 (10) &{\cellcolor[gray]{0.9}} 14 (3) &{\cellcolor[gray]{0.9}} 65 (3) &  {\cellcolor[gray]{0.9}} 560 (50)\\
&&\cellcolor[gray]{0.8} ADGD &\cellcolor[gray]{0.8} 14 (10)&\cellcolor[gray]{0.8}  15 (3) &\cellcolor[gray]{0.8} 64\phantom{0}(3) &\cellcolor[gray]{0.8} 220 (3)\\ 
&& Tikhonov & 8.0 (7) &  15 (2) & 63 (4) & 990 (300)\\ 
\hhline{|~------|}
&\multirow{3}{2em} {60}&  \cellcolor[gray]{0.9 } DGD & \cellcolor[gray]{0.9} 40 (10) &\cellcolor[gray]{0.9} 33 (4) &\cellcolor[gray]{0.9} 77 (3) &\cellcolor[gray]{0.9} 560 (10)\\
&&\cellcolor[gray]{0.8}  ADGD &\cellcolor[gray]{0.8}  40 (20) &\cellcolor[gray]{0.8} 33 (6) &\cellcolor[gray]{0.8}  77 (3) &\cellcolor[gray]{0.8}  220 (3)\\ 
&&Tikhonov & 35 (30) & 36 (8) & 78 (2) & 1700 (500)\\ 
\hline
\end{tabular}}
\end{center}}
 \end{table}
 
 \subsection{Matrix completion}
We consider  the problem of recovering a low-rank data matrix $W\in\RR^{n\times p}$ from a sampling of its entries.
We denote by $\Omega$ the subset of indices corresponding to sampled entries.
We find an approximate solution of this problem  by minimizing a strongly convex relaxation \cite{CaiCanShe10} given 
by the sum of the nuclear norm with the squared Frobenius norm, that is:
\begin{equation}
\label{pb:mc}
\min_{\mathcal{X}W=\widehat{Y}} \|W\|_*+\frac{\alpha}{2}\|W\|_F^2,
\end{equation}
where $\widehat{Y}\in\RR^{n\times p}$, is such that, for every $(i,j)\not\in\Omega$, $\widehat{Y}_{i,j}=0$, and $\mathcal{X}\colon\RR^{n\times p}\to \RR^{n\times p}$ is such that $(\mathcal{X}W)_{i,j}=W_{i,j}$ if $(i,j)\in\Omega$ and $0$ otherwise. 
DGD applied to this problem is the Singular Value Tresholding (SVT) algorithm described in \cite{CaiCanShe10} and note that, interestingly, 
ADGD is its accelerated counterpart. The most expensive computational part is the proximal step, which requires
an SVD decomposition \cite{ComPes11}. While in \cite{CaiCanShe10} the authors apply the algorithm to noisy data, they then propose as
an improvement a different relaxation \cite{CanPla09}. Here we show that SVT with early stopping is indeed an efficient algorithm to deal with matrix completion of
noisy data.  We tested the performance on simulated data using a standard procedure described in \cite{CaiCanShe10}. We multiplied random gaussian 
matrices with independent entries and variance $1$ of size $n\times r$ and $r\times p$ where $r$ is the chosen rank,  and then we added an 
additive gaussian noise. We computed the Root Mean Square Error of the proposed approximation:
${\rm RMSE(\widehat{W})}={ (\sum_{(i,j)\in A} (\widehat{W}_{i,j}-Y_{i,j})^2)^{1/2}}/{|A|}$ ,
where $A$ is the test set. 
 \begin{table}[t]
 { \caption{\small Minimal achieved RMSE and associated cost of  ADGD and Tikhonov approach solved with warm-starting, accelerated proximal gradient method, on simulated data with additive gaussian noise of standard deviation $\delta$.  We used the ground truth to select the best parameter. The percentage of known entries (knowledge ratio) is 0.12, 0.39 and 0.57 for, respectively, ranks 10, 50 and 100. The matrices are of size 1000x1000.  The results were averaged over 5 simulations with the standard deviation between parentheses. For ADGD, the iterative trials were capped at 500 iterations (for noise levels of 0.01) and 250 (noise of 0.1 and 1).}
  \label{mc-table}
\begin{center}
{\footnotesize
\begin{tabular}{ cccccc } 
\hline
\multirow{2}{2em}{Noise} & \multirow{2}{2em}{Rank}  & RMSE & RMSE  &Iterations&Iterations \\
&   & ADGD  & Tikhonov  &ADGD & Tikhonov \\[0.3ex]
\hline
\multirow{3}{2em}{0.01}& \cellcolor[gray]{0.9} 10 &\cellcolor[gray]{0.9} $2.1\cdot10^{-3} (3.8\cdot10^{-5})$&\cellcolor[gray]{0.9} $7.6\cdot10^{-3} (1.5\cdot 10^{-4})$ &\cellcolor[gray]{0.9} 500 (0) &\cellcolor[gray]{0.9} 527 (3)\\
& \cellcolor[gray]{0.8} 50 &\cellcolor[gray]{0.8} $3.2\cdot10^{-3} (2.2\cdot10^{-5})$&\cellcolor[gray]{0.8} $9.1\cdot10^{-3} (4.1\cdot 10^{-5})$ &\cellcolor[gray]{0.8} 500 (0) &\cellcolor[gray]{0.8} 295 (1)\\
& 100 & $4.7\cdot10^{-3} (2.9\cdot10^{-5})$ & $1.1\cdot10^{-2} (6.4\cdot10^{-5})$ & 500 (0) & 273 (1)\\
\hline
\multirow{3}{2em}{0.1}& \cellcolor[gray]{0.9} 10 &\cellcolor[gray]{0.9} $0.23 (4.6\cdot10^{-3})$&\cellcolor[gray]{0.9} $0.75 (9.0\cdot 10^{-3})$ &\cellcolor[gray]{0.9} 250 (0) &\cellcolor[gray]{0.9} 539 (3.7)\\
& \cellcolor[gray]{0.8} 50 &\cellcolor[gray]{0.8} $0.35 (2.1\cdot10^{-3})$&\cellcolor[gray]{0.8} $0.95 (2.7\cdot 10^{-3})$ &\cellcolor[gray]{0.8} 250 (0) &\cellcolor[gray]{0.8} 425 (0.49)\\
& 100 & $0.48 (2.0\cdot10^{-3})$ & $1.1 (4.3\cdot10^{-3})$ &190 (0) & 470 (0.4)\\
\hline
\multirow{3}{2em}{1}& \cellcolor[gray]{0.9} 10 &\cellcolor[gray]{0.9} $27 (0.28)$&\cellcolor[gray]{0.9} $76 (1.1)$ &\cellcolor[gray]{0.9} 191 (0) &\cellcolor[gray]{0.9} 698 (3.6)\\
& \cellcolor[gray]{0.8} 50 &\cellcolor[gray]{0.8} $41 (0.20)$&\cellcolor[gray]{0.8} $108 (0.20)$ &\cellcolor[gray]{0.8} 205 (0.4) &\cellcolor[gray]{0.8} 729 (3.6)\\
& 100 & $$55 (0.18)$$ & $125 (0.11)$ & 210 (0.4) & 742 (2.8)\\
\hline
\end{tabular}}
\end{center}}
 \end{table}
 As can be seen in Table~\ref{mc-table} ADGD is comparable to state of the art Tikhonov regularization, with a significantly
lower computational cost. In addition, we compare DGD with Tikhonov regularization
 (with accelerated proximal gradient+warm restart) on the MovieLens 100k dataset\footnote{http://grouplens.org/datasets/movielens/}.  
 We averaged our results over five trials.  We left out one tenth of the known entries at each trial and chose the best step/parameter  via 2-fold 
 cross validation. The mean RMSE for DGD  and Tikhonov was 1.02. It required 250 iterations on average using DGD,  and 550 iterations using Tikhonov.
\begin{figure}
\begin{center}
\subfigure{\includegraphics[width=0.151\linewidth]{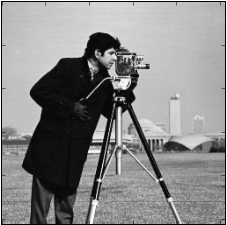}}
\subfigure{\includegraphics[width=0.15\linewidth]{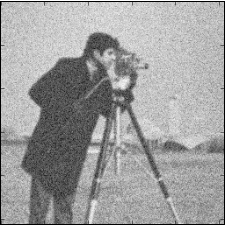}\label{fig:noisy}}
\subfigure{\includegraphics[width=0.151\linewidth]{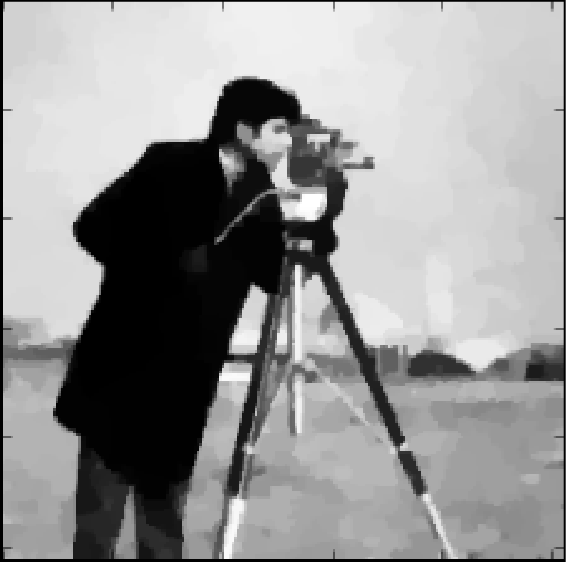}\label{fig:43}}
\subfigure{\includegraphics[width=0.151\linewidth]{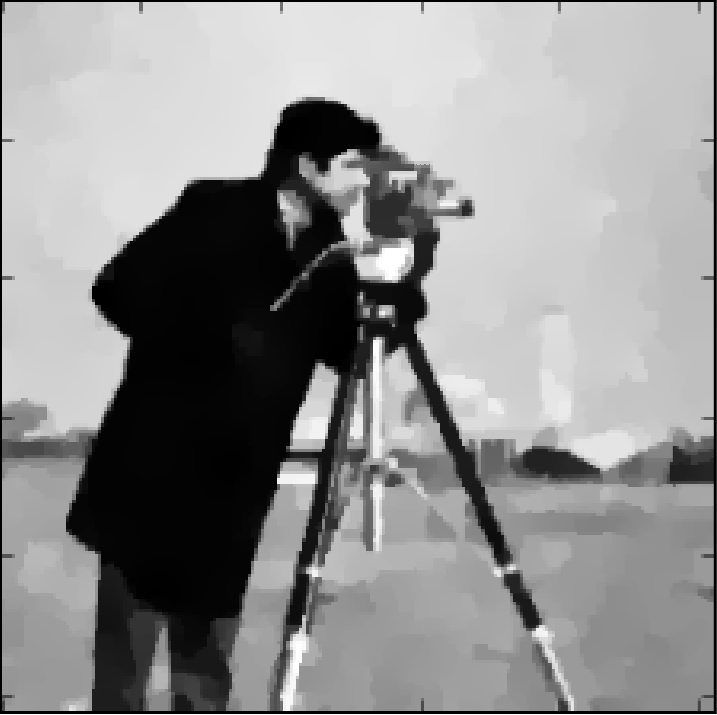}\label{fig:2}}

\caption{\small From left to right: orginal Cameraman image, noisy blurred image,  restored image with Tikhonov regularization, restored image with  ADGD.}
\label{fig:cman}
\end{center}
\end{figure}

 \subsection{Image  deblurring}
Finally, we apply ADGD to an image processing problem, namely 
 deblurring, with a strongly convex perturbation of total variation. 
More precisely, given an image $W\in \RR^{256\times 256}$, we consider the regularization function
$ R(W)=TV(W)_{1,2}+ \frac{3}{2} \|W\|^2$,
 where $TV$ is the discrete total variation.
 In this application the proximity operator of the total variation penalty is not available in closed
 form. In our experiments, this is computed at each iteration using  20 steps of accelerated dual 
 forward backward on the denoising problems corresponding to \eqref{e:prox}, and by warm starting with 
 the previous approximate proximal point.
We assume to have access to a noisy image $\widehat{y}$, obtained corrupting the original
image with a Gaussian blur of one pixel and an additive Gaussian noise with variance $0.01$.
 We compared the iterative regularization ADGD with early stopping with
 the solution obtained with the Tikhonov approach corresponding to the best regularization parameter
 on the cameramen image. The quality of an approximation of the original image is measured in terms of 
 PSNR, and the best results are reported in Figure~\ref{fig:cman}. 
On the computational side, for the Tikhonov approach we set $\lambda_0=10^5$, and
 then decreased it by multiplying it  by $0.8$ at each step. The best solution is obtained for $\lambda=6.8$, 
while  iterative regularization achieves the best results at the third iteration.

{\small {\bf Acknowledgments}  This material is based upon work supported by the Center for Brains, Minds and Machines (CBMM), funded by NSF STC award CCF-1231216.
L. R. acknowledges the financial support of the Italian Ministry of Education, University and Research FIRB project RBFR12M3AC.
S. V. acknowledges the support of the Gnampa-Indam project 2017:``Algoritmi di ottimizzazione ed equazioni di evoluzione ereditarie''.}

\newpage

\appendix
\section{Supplementary Material}
\subsection{Derivation of the algorithm}
We start showing that the proposed procedures DGD and ADGD are indeed a gradient and an accelerated gradient descent algorithm
applied to the dual problem of the noisy minimization problem
\begin{equation}
\label{e:noisya}
\min_{Xw=\widehat{y} }  R(w), \quad \text{with } R=F+\frac{\alpha}{2}\|\cdot\|^2.
\end{equation}
Let $C$ be a convex and closed subset of $\mathbb{R}^n$. With $\delta_C$ we denote the indicator function of $C$, which takes
value $0$ on $C$ and $+\infty$ otherwise. 
The optimization problem in \eqref{e:noisya} can be equivalently written as
\begin{equation}
\label{e:noisya2}
\min_{w\in\mathbb{R}^p} R(w)+ \delta_{\widehat{y}} (Xw).
\end{equation}
The above optimization problem is given by the sum of two convex, proper, and lower semicontinuous functions, where one of the
two is composed with a linear operator. This is the suitable form to apply Fenchel-Rockafellar duality. First recall that
the Fenchel conjugate of $G\colon\mathbb{R}^p\to[-\infty,+\infty]$ is  $G^*\colon\mathbb{R}^p\to[-\infty,+\infty]$, such that, for
every $u\in\mathbb{R}^p$, $G^*(v)=\sup_{w\in\mathbb{R}^p} \langle v,w\rangle- G(w)$.    The dual of the problem in~\eqref{e:noisya2}
is then (see e.g. \cite[Definition 15.19]{livre1})
\begin{equation}
\label{pb:dual}
\min_{v\in\mathbb{R}^p} R^*(-X^Tv)+\langle \widehat{y},v\rangle.
\end{equation}

Since $R=F+(\alpha/2)\|\cdot\|^2$, its conjugate is differentiable with Lipschitz continuous gradient and is given by (see \cite[Example 13.4]{livre1}), 
$$
R^*(v)=\frac{1}{2\alpha}\|v\|^2-\inf_{u\in\mathbb{R}^p} \Big\{F(u)+\frac{\alpha}{2}\big\|u-\frac{v}{\alpha}\big\|^2 \Big\}.
$$ 
The second term on the right hand side is called Moreau envelope of $F$, and we use the notation
$^{\alpha^{-1}\!\!}F(v)=\inf_{u\in\mathbb{R}^p} \Big\{F(u)+\frac{\alpha}{2}\big\|u-\frac{v}{\alpha}\big\|^2 \Big\}$.
A formula for the gradient of $^{\alpha^{-1}\!\!}F$ is known \cite[Proposition 12.29]{livre1}, and we derive 
\begin{align}
\nonumber \nabla R^*(v)& =\alpha^{-1}{v}- \alpha^{-1} \nabla (^{\alpha^{-1}\!\!}F)(\alpha^{-1} v)\\
\nonumber&=\alpha^{-1}{v}- \alpha^{-1} \Big(\alpha\big(\alpha^{-1} v -\prox_{\alpha^{-1} F}(\alpha^{-1}v) \big) \Big)\\
&=\prox_{\alpha^{-1} F}(\alpha^{-1}v)
\end{align}
This implies that one step of gradient descent applied to the problem in~\eqref{pb:dual} can be written as
\[
v_{t+1}=v_t+\gamma \big(X \prox_{\alpha^{-1} F}(-\alpha^{-1}X^T v_t)-\widehat{y}\big),
\]
and this is the main iteration in  DGD.  The derivation of ADGD is analogous, simply the gradient descent method is replaced by FISTA 
acceleration~\cite{BecTeb09}. 
\subsection{Proofs and auxiliary results}
In order to prove Theorems~\ref{thm:dgd} and \ref{thm:adgm}, we need some auxiliary results. 
We start with dual gradient descent.
\begin{theorem} 
\label{thm:dupri}
Let $(w_t)_{t\in\NN}$ be the sequence in $\RR^p$ generated by iteration~\eqref{e:main1} and define
$u_t=\sum_{k=0}^tw_k/(t+1)$.
Assume that there exists $\bar{v}\in\RR^p$ such that
\[
-X^T\overline{v}\in\partial R(w^\dag).
\]
Then
\begin{align}
\label{e:primal_rate}
\nonumber \|w_t-w^\dagger\| \leq \frac{\|X\|\|v^\dagger\|}{\alpha\sqrt{t}}\\
\|u_t-w^\dag\|\leq \frac{\|X\| \|v^\dag\|}{\alpha \sqrt{t}},
\end{align}
where $v^\dag$ is a solution of the dual problem. 
\end{theorem}

\begin{proof}Thanks to the assumption $-X^T\overline{v}\in\partial R(v^\dag)$ , strong duality holds, namely
the dual problem has a solution $v^\dagger$, and the minimum of the problem in~\eqref{pb:dual} is the same as the minimum of the problem in~\eqref{e:noisy}.
For every $v\in\mathbb{R}^p$, let $D(v)=R^*(-X^Tv)+\langle \widehat{y},v\rangle$. 
Then (see for instance \cite[Theorem 3.1]{BecTeb09}) it holds
\[
D(v_t)-D(v^\dagger)\leq \frac{\|X\|^2 \|v_0-v^\dagger\|^2}{2\alpha t}.
\]
Next, strong convexity implies that 
\[
\frac{\alpha}{2}\|w_t-w^\dagger\|^2 \leq D(v_t)-D(v^\dagger).
\]
Combining the two inequalities, and recalling that $v_0=0$, we derive
\[
\|w_t-w^\dagger\| \leq \frac{\|X\|\|v^\dagger\|}{\alpha\sqrt{t}}.
\]
The statement then follows from convexity of the norm.
\end{proof}
And now we are ready for the main stability result.

\begin{proposition}
\label{l:s1} 
Let $(\widehat{w}_t)_{t\in\NN}$, $(\widehat{u}_t)_{t\in\NN}$  be  the
sequences generated by DGD. Let $(w_t)_{t\in\NN}$ be defined as in  \eqref{e:main1} with $v_0=0$, and define, for every $t\in\mathbb{N}$, $u_t=\sum_{k=0}^t w_k/(t+1)$. 
Assume $\delta<1$.
Then the following hold:
\begin{itemize}
\item[i)] There exists $t_\delta\in \{\lfloor1/\delta\rfloor,\ldots, 2\lfloor1/\delta\rfloor\}$ such that
\begin{equation}
\label{e:aq34}
\| w_{t_\delta}-\widehat{w}_{t_\delta} \|  \leq 2\|X\|^{-1} \delta t_\delta^{1/2}.
\end{equation}
\item[ii)] For every $t\in\NN$,
\begin{equation}
\label{e:aqavgs}
\| u_t-\widehat{u}_t \|  \leq 2\|X\|^{-1} \delta t^{1/2}.
\end{equation}
\end{itemize}
\end{proposition}
\begin{proof} \ \\
i): For every $t\in\NN$, using the firm nonexpansiveness of $\prox_{\alpha^{-1}F}$ and the definition of $\gamma$
\begin{alignat}{2}
\label{e:str}
\| \widehat{v}_t-v_t+\gamma (X(\widehat{w}_t -w_t)\|^2
&= \|\widehat{v}_t-v_t \|^2 + 2\gamma \langle{\widehat{v}_t-v_t},{X (\widehat{w}_t -w_t)}\rangle +
\gamma_{t}^2 \| X (\widehat{w}_t -w_t)\|^2\notag\\
 &\leq \|\widehat{v}_t-v_t \|^2 -2\gamma \alpha \| \widehat{w}_t -w_t\|^2+ \gamma^2 \| X(\widehat{w}_t -w_t)\|^2\notag\\
&\leq \|\widehat{v}_t-v_t \|^2 - \gamma\alpha\| \widehat{w}_t -w_t\|^2\notag\\
&\leq \|\widehat{v}_t-v_t \|^2.
\end{alignat}
Consequently, 
\begin{align*}
\|\vhat_{t+1}-v_{t+1}\|&=\|\vhat_{t}-v_t+\gamma X(\what_t-w_t)-\gamma(\widehat{y}-y)\|\\
&\leq \|\vhat_{t}-v_t\|+\gamma\delta
\end{align*}
and therefore
\begin{equation}
\label{e:tb}
\|\vhat_{t+1}-v_{t+1}\|\leq \gamma\delta(t+1).
\end{equation}
Moreover,
\begin{align*}
 \|\vhat_{t}-v_t+\gamma X(\what_t-w_t)\|^2&=\|\vhat_{t+1}-v_{t+1}+\gamma(\widehat{y}-y)\|^2\\
&= \|\vhat_{t+1}-v_{t+1}\|^2+\gamma^2\|y-\widehat{y}\|^2+2\gamma\langle \vhat_{t+1}-v_{t+1},\widehat{y}-y \rangle\\
&\geq \|\widehat{v}_{t+1}-v_{t+1} \|^2+\gamma^2\|y-\widehat{y}\|^2-2\gamma \delta\|\vhat_{t+1}-v_{t+1}\|.
\end{align*}
Hence, \eqref{e:str} and \eqref{e:tb} yield
\begin{align}
\nonumber\gamma\alpha\| \widehat{w}_t -w_t\|^2&\leq \|\widehat{v}_t-v_t \|^2- \|\vhat_{t}-v_t+\gamma X(\what_t-w_t)\|^2\\
\nonumber&\leq  \|\widehat{v}_t-v_t \|^2-\|\widehat{v}_{t+1}-v_{t+1} \|^2+2\gamma \delta\|\vhat_{t+1}-v_{t+1}\|\\
\label{e:tele}&\leq  \|\widehat{v}_t-v_t \|^2-\|\widehat{v}_{t+1}-v_{t+1} \|^2+2\gamma^2 \delta^2(t+1).
\end{align}
Summing the previous inequality for $t\in\{t_1,\ldots,T\}$, with $T\geq 2$ we derive
\begin{equation}
\gamma\alpha\sum_{t=t_1}^{T} \| \widehat{w}_t -w_t\|^2 \leq 4\gamma^2\delta^2 T^2
\end{equation}
Taking $t_1=\lfloor1/\delta \rfloor$ and $T=2\lfloor1/\delta \rfloor$ it follows that
\[
\sum_{t=t_1}^{T} \| \widehat{w}_t -w_t\|^2 \leq 4 \|X\|^{-2} \delta^2 \lfloor1/\delta \rfloor^2.
\]
Thus there exists at least a $t_\delta\in\{\lfloor1/\delta \rfloor,\ldots, 2\lfloor1/\delta \rfloor\}$ such that
\[
 \|\widehat{w}_{t_\delta} -w_{t_\delta}\|^2 \leq 4 \|X\|^{-2} \delta^2 \left\lfloor\frac1{\delta}\right \rfloor\leq 4 \|X\|^{-2} \delta^2 t_\delta.
\]
ii): Summing the inequalities in~\eqref{e:tele} for $t=0,\ldots,T$ we derive:
\begin{equation}
\gamma\alpha\sum_{t=0}^{T} \| \widehat{w}_t -w_t\|^2 \leq 4\gamma^2\delta^2 T^2
\end{equation}
Convexity of $\|\cdot\|^2$, and the fact that $\widehat{w}_0=w_0$ imply
\begin{equation}
\|\widehat{u}_T -u_T\|^2 \leq\frac{1}{T+1}\sum_{t=0}^{T} \| \widehat{w}_t -w_t\|^2 \leq 4\|X\|^{-2}\delta^2 T
\end{equation}
\end{proof}

The following lemma characterizes the asymptotic behavior of the sequence $(\theta_t)_{t\in\NN}$.
\begin{lemma}\label{lem:theta} Let $(\theta_t)_{t\in\NN}$ be the sequence defined in~ADGD.
Then, for every $t\in\NN$ 
\[
\frac{t+1}{2}\leq\theta_t \leq t+1
\]
\end{lemma}
\begin{proof}
We prove the first inequality by induction. The case $t=0$ is clear since $\theta_0=1$. Now suppose that
the inequality is true for $t$.  We derive
\[
\theta_{t+1}=\frac{1+\sqrt{1+4\theta_t^2}}{2}\geq  \frac{1+\sqrt{1+(t+1)^2}}{2} \geq \frac{t+2}{2}.
\]
For the second inequality, the case $t=0$ is also clear. Now suppose that the inequality is true for $t$.
We derive
\[
\theta_{t+1}=\frac{1+\sqrt{1+4\theta_t^2}}{2}\leq  \frac{1+\sqrt{1+4(t+1)^2}}{2}\leq \frac{1+1+2(t+1)}{2}=t+2.
\]
\end{proof}

The following theorem is obtained exploiting existing results on convergence of forward-backward algorithm in the presence 
of computational errors. 
In particular, the result is derived combining \cite[Proposition 3.3]{AujDos15} (see also \cite{SchLerBac11,VilSal13} for related
results) with Lemma~\ref{lem:theta} and the relationship between convergence of the dual objective function and the primal iterates.
\begin{theorem}
\label{l:s2}
Let  $(\widehat{w}_t)_{t\in\NN}$  be  the
sequence generated by ADGD.
Then, for every $t\in\NN$, $t\geq 1$, 
\begin{equation}
\label{e:aq100}
 \| \what_t-w^\dag\|  \leq \frac{2\|X\|\|v^\dag\|}{\alpha t}+4\|X\|^{-1} \delta t.
\end{equation}
\end{theorem} 
Theorem~\ref{thm:dgd} is a direct corollary of Theorem~\ref{thm:dupri} and Proposition~\ref{l:s1}.
Theorem~\ref{thm:adgm} directly follows from Theorem~\ref{l:s2}.
 \begin{proof}
For every $v\in\mathbb{R}^p$, let $D(v)=R^*(-X^Tv)+\langle \widehat{y},v\rangle$. Then strong convexity yields
\begin{equation}
\label{e:strconv}
(\forall t\in\NN) \qquad \frac{\alpha}{2}\|\what_t-w^\dag\|^2  \leq D(\vhat_t)-\min_{v\in\RR^p} D(v).
\end{equation}

Proposition 3.3 in \cite{AujDos15} and Lemma~\ref{lem:theta} imply
\begin{align*}
D(\vhat_t)-\min_{v\in\RR^p} D(v)& \leq \frac{1}{2\gamma\theta^2_t} \left(\|v^\dagger\|+\gamma\delta\sum_{k=0}^t\theta_k\right)^2\\
& \leq \frac{1}{\gamma(t+1)^2}\left(\|v^\dagger\|+\gamma\delta \frac{(t+2)(t+3)}{2}\right)^2
\end{align*}
we derive
\begin{equation}
 \| \what_t-w^\dag\|  \leq \frac{{2}\|X\|\|v^\dagger\|}{\alpha t} +\frac{4}{\|X\|}\delta t
\end{equation}
 \end{proof}
We are now ready to prove the main results.
\begin{proof}[Proof of Theorem~\ref{thm:dgd}]. 
Theorem~\ref{thm:dupri} and Proposition~\ref{l:s1} imply
\[
\|\widehat{u}_t-w^\dagger\| \leq a t^{1/2}\delta+b t^{-1/2}. 
\]
Since $t_\delta=\lceil c\delta^{-1}\rceil$, we have $c\delta^{-1}\leq t_\delta\leq c\delta^{-1}+1$, therefore
\[
\|\widehat{u}_t-w^\dagger\| \leq a t^{1/2}\delta+b t^{-1/2}\leq a (c\delta^{-1}+1)^{1/2}\delta+ bc^{-1/2}\delta^{1/2}. 
\]
The statement follows noting that $(c\delta^{-1}+1)^{1/2}\leq (c^{1/2}+1) \delta^{-1/2}$.
\end{proof}
Finally,w e proove Theorem~\ref{thm:adgm}.
\begin{proof}[Proof of Theorem~\ref{thm:adgm}]. 
Theorem~\ref{l:s2} yields
\[
\|\widehat{w}_t-w^\dagger\| \leq a t\delta+b t^{-1}. 
\]
Since $t_\delta=\lceil c\delta^{-1/2}\rceil$, we have $c\delta^{-1/2}\leq t_\delta\leq c\delta^{-1/2}+1$, therefore
\[
\|\widehat{u}_t-w^\dagger\| \leq a t \delta+b t^{-1}\leq a (c\delta^{-1/2} +1)\delta+ bc^{-1}\delta^{1/2}. 
\]
The statement follows noting that $(c\delta^{-1/2}+1)\leq (c+1) \delta^{-1/2}$.
\end{proof}
\end{document}